\documentclass{amsart}
\usepackage{graphicx}
\usepackage{color}
\usepackage{subfig}
\usepackage{amsfonts,amssymb,amscd,amsmath,enumerate,verbatim,newlfont,calc}

 \usepackage{amsfonts}

 \newtheorem{theorem}{Theorem}[section]
\newtheorem{lemma}[theorem]{Lemma}
\newtheorem{proposition}[theorem]{Proposition}
\theoremstyle{definition}
\newtheorem{definition}[theorem]{Definition}
\newtheorem{example}[theorem]{Example}

 \theoremstyle{remark}
\newtheorem{remark}[theorem]{Remark}
\numberwithin{equation}{section}
\begin{document}

 \title[Lorentzian non-stationary dynamical systems ]{Lorentzian non-stationary dynamical systems }
\author[  MohammadReza Molaei and  Najmeh Khajoei ]{{ $^{\mathrm{a}}$MohammadReza Molaei and $^{\mathrm{a,b}}$Najmeh Khajoei  
}\\$^{\mathrm{a}}${Mahani Mathematical Research Center Shahid Bahonar University of Kerman, Kerman, 76169-14111, Iran.\\$^{\mathrm{b}}$Young Research Society, Shahid Bahonar University of Kerman,
Kerman, 76169-14111, Iran.\\e-mails:  mrmolaei@uk.ac.ir, khajuee.najmeh@yahoo.com}}

\maketitle

\begin{abstract}
In this paper, we introduce a Lorentzian Anosov family ($LA$-family) up to a sequence of distributions of null vectors. We prove for each $p\in M_i$, where $M_i$ is a Lorentzian manifold for $i\in\mathbb{Z}$ the tangent space $M_i$ at $p$ has a unique splitting and this splitting varies continuously on a sequence via the distance function created by a unique torsion-free semi-Riemannian connection. We present three examples of $LA$-families. Also, we define Lorentzian shadowing property of type $I$ and $II$ and prove some results related to this property.
\end{abstract}
\noindent AMS subject classification: 37D05, 37D20, 53B30, 37C50\\
{Keywords:} Lorentzian Anosov family; Semi-Riemannian connection;
Lorentzian solenoid; Lorentzian shadowing property of type $I$ and $II$
\section{Introduction}
%Anosov families as a generalization of Anosov diffeomorphisms were introduced in $2005$ by P. Arnoux and A. M. Fisher \cite{n5}. In fact Anosov family is a non-stationary dynamical system $\{f_i\}_{i\in\mathbb{Z}} $ defined on a sequence of compact Riemannian manifolds $\{M_i\}_{i\in\mathbb{Z} }$. 
%%%%%%%%%%%%%%%%%%%%%%%%%%%%%%%%%%%%
There are some fantastic results on Riemannian manifolds which can be extendable for Lorentzian manifolds. For instance, in \cite{r1} Porwal and Mishra introduced a new class of geodesic local E-convex sets and geodesic local (semilocal) E-convex functions on Riemannian manifolds and studied their features and in \cite{r2} pointwise slant submersions from almost product Riemannian manifolds to Riemannian manifolds were investigated by Sepet and Ergut. Also, Anosov family as a generalization of Anosov diffeomorphisms were introduced in $2005$ by P. Arnoux and A. M. Fisher \cite{n5} is a non-stationary dynamical system $\{f_i\}_{i\in\mathbb{Z}} $ defined on a sequence of compact Riemannian manifolds $\{M_i\}_{i\in\mathbb{Z} }$. Since we are enthusiastic about the dynamical system on Lorentzian manifolds then the concept of Anosov families on Lorentzian manifolds is extended.  
%%%%%%%%%%%%%%%%%%%%%%%%%%%%%%%%%%%%%%%
  For this purpose we take a sequence of Hausdorff Lorentzian manifolds $\{M_i\}_{i\in\mathbb{Z} }$ of the same dimension $m$ with Lorentzian metrics $\{g_i(. , .)\}_{i\in\mathbb{Z}}$.
We recall that a Lorentzian
metric $g_i$ for a manifold $M_i$ as a tensor field of type $(0, 2)$ with the diagonal form $(-, +, ..., +)$ at each point $p\in M_i$ \cite{n8, n9}.
Nonzero tangent vectors are classified as timelike, spacelike, nonspacelike,
or null according to whether $g_i(v,v)<0, ~g_i(v,v)>0,~g_i(v,v)\leqslant 0$ or $g_i(v,v)=0$ \cite{n7}. We take $M$ as a disjoint union of $M_i$ i.e. $M=\displaystyle\bigcup_{i\in\mathbb{Z}}^0=\{(i,m_i)~ \vert~ m_i\in M_i~ ,~ i\in\mathbb{Z}\}$. The topology of $M$ is the set $\tau = \{ \{i\} \times V_i ~\vert ~ V_i ~\text{is an open set in}~ M_i\}$ and it's atlas is the set of charts of the form $A=\{(\tilde{U}_i, \tilde{\phi}_i) \}$ such that $\tilde{U}_i=\{i\}\times U$ and $\tilde{\phi}_{\tilde{U_i}} : \tilde{U}_i\rightarrow \mathbb{R}^m$ is defined by $(i, m)\mapsto \phi_{U}(m)$ where $(U, \phi_U)$ is a chart for $M_i$ and $i\in\mathbb{Z}$. Therefore the inclusion $i: M_{i}\hookrightarrow M$ is a diffeomorphims
onto its image. So
each subset of $M_i$ is recognized with a subset of $M$, hence each chart on $M_i$ gives a chart on
$M$. These charts are expanded by the former method to a maximal atlas for $M$.
The construction of $M$ as a manifold implies that if the image of a curve $\gamma : (-\epsilon, \epsilon)\rightarrow M$ is in a chart $(\tilde{U_i}, \phi_{\tilde U_i})$, then there is a unique $i_0\in \mathbb{Z}$ such that $\gamma(-\epsilon, \epsilon)\subseteq M_{i_0}$.

 In section $3$, we extend Anosov family to a Lorentzian Anosov family up to a sequence of distribution $p\mapsto E_i^n(p)$, where $ E_i^n(p) $ is null subspace of the tangent space of the ambient Lorentzian manifold $M_i$ that is an extension of hyperbolic autonomous discrete dynamical system on semi-Riemannian manifold \cite{n1} to non-autonomous one. In proposition \ref{p1} we prove the tangent space of $M$ at a given $p$ has a unique splitting. In theorem \ref{t1} we determine the form of splitting.\\
In section $4$ we provide two examples of Lorentzian Anosov families. In section $5$ by using of the results of \cite{n3} for non-autonomous discrete dynamical system and induced dynamics in hyperspace of them, we define Lorentzian shadowing property of type $I$ and type $II$ on $L$-family $(M, \langle . , .\rangle, F)$ and on the induced dynamic by it. In theorem \ref{u2} we prove shadowing property of type $I$ is invariant by uniformly conjugacy and the product of two $L$-families has shadowing property of type $I$ if and only if they have shadowing property of type $I$. In proposition \ref{u3} we prove if $L$-family $(M, g, F)$ has shadowing property of type $ II $ then it has shadowing property of type $ I $.\\
\section{Preliminaries}
We begin this section by recalling the concepts of connections and parallel translation in semi-Riemannian geometry which the Lorentzian geometry is a special case of it \cite{n7}.\\
Let $\chi(M)$ denote the set of all smooth vector fields defined on $M$ and let $\mathcal{F}(M)$ denote the ring of all smooth real-valued functions on $M$. A connection is a mapping $\nabla : \chi(M) \times \chi(M)\rightarrow\chi(M)$ with the following properties \\
$i)~ \nabla_{V}(X+Y)=\nabla_V X +\nabla_V Y;$\\
$ii)~ \nabla_{f V+h W} (X) =f \nabla_V (X) +h\nabla_V(Y);$\\
$iii) ~\nabla_V(f X)=f \nabla_V X+ V(f) X;$\\
for all $f, h\in\mathcal{F}(M)$ and all $X, Y, V, W\in \chi(M)$.\\
The vector $\nabla_{X(p)} Y= \nabla_{X} Y\vert_{p}$ at a point $p\in M$ depends only on the value $X(p)=X_p$ of $X$ at $p$ and the values of $Y$ along any smooth curve which passes through $p$ which has the velocity $X(p)$ at $p$. Let the connection $\nabla$ on $M$, a curve $\gamma : [-\epsilon, \epsilon]\rightarrow M$ and a smooth vector field $Y$ along $\gamma$ be given. Then for $t_0\in [-\epsilon, \epsilon]$ we can locally extend $Y$ to a smooth vector field defined on a neighborhood of $\gamma(t_0)$. We denote the covariant derivative of $Y\in\chi(M)$ along $\gamma$ by $\dfrac{D Y}{dt}$ and it is defined by $\nabla_{\gamma'(t)} Y$.
A vector field $Y$ along $\gamma$ which satisfies $\nabla_{\gamma' (t)} Y(t)=0$ for all $t\in [-\epsilon, \epsilon]$ is called a parallel vector field along $\gamma$. If $v\in T_p M$ and $\gamma : (-\epsilon, \epsilon)\rightarrow M$ is a smooth curve passing through $p$, that is, $\gamma(0)=p$, then it is proved that there is a unique parallel vector field $Y$ along $\gamma$ with $Y_p=v$. The mapping $P_t : T_p M\rightarrow T_{\gamma(t)} M$, $v\mapsto Y_{\gamma(t)}$ is called a parallel transition.\\
The torsion tensor $T$ of $\nabla$ is the mapping $T: \chi(M)\times \chi(M)\rightarrow\chi(M)$ defined by $T(X, Y)=\nabla_X Y -\nabla_Y X-[X, Y]$. A connection $\nabla$ with $T=0$ is said to be torsion free or symmetric. \\
Two vectors $v, w$ in $T_p M$ are orthogonal if $g(v, w)=0$. A given vector $v\in T_p M$ is said to be a unit vector if $\vert g(v, v)\vert=1$. Thus an orthonormal basis $\{e_1, e_2, ..., e_n\}$ of $T_p M$ satisfies $\mid g(e_i, e_j)\vert=\delta_i^j$.
Let $(M,g)$ be an n-dimensional manifold $M$ with a semi-Riemannian metric $g$ of arbitrary signature $(-,..., -, +, ..., +)$. There exists a unique connection $\nabla$ on $M$ such that \\
$i)$ $Z(g(X,Y))=g(\nabla_Z X, Y)+g(X, \nabla_Z Y)$
and\\
$ii)$ $\nabla_X Y-\nabla_Y X=[X, Y]$ for all $X, Y, Z \in\chi(M)$. This connection is called the Levi-Civita connection of $(M,g)$. The condition $(i)$ means that the connection $\nabla$ is compatible with the metric $g$ and condition $(ii)$ means that $\nabla$ is torsion free. By replacing $Z=\gamma '$ in $(i)$, we find a unique parallel translation of vector fields along a given smooth curve $\gamma$ of $M$ which preserves $g$ \cite{n7}.\\
As stated in \cite{n1} we use parallel translation to define a distance function $d$ on the subspaces of the tangent spaces. Let $\gamma : (-\epsilon, \epsilon)\rightarrow M$ be a smooth curve passing through $p$. Then
$$d(u, B_E)= min\{\mid g_{\gamma(\zeta )}(P_{\zeta -t}(u) -w, P_{\zeta -t}(u) -w) \mid : w \in B_E\},$$
where $u\in T_{\gamma(t)} M$, $E$ is a subspace of $T_{\gamma(\zeta )} M$ with the basis $B_E$ and $t, \zeta \in (-\epsilon, \epsilon)$. For two given subspaces $E$ and $F$ of $T_{\gamma(\zeta )} M$ and $T_{\gamma(t)} M$ with the basis $B_E$ and $B_F$ we define the distance function $d(B_E, B_F)$ by
\begin{equation}\label{e1}
d(B_E, B_F)=max\{ max\{d(v, B_F) : v\in B_E \} , max\{ d(u, B_E) : u\in B_F \} \}
\end{equation}

 \section{Lorentzian Anosov family}
In this section $M$ is the disjoint union of the Hausdorff Lorentzian manifolds $\{M_i\}_{i\in\mathbb{Z}}$, with Lorentzian metrics $\{g_i\}$.
We can define the Lorentzian metric $g=\langle . , .\rangle$ on $M$ by $\langle . , .\rangle\vert_{M_i} =\langle . , .\rangle_i$ for each $i\in \mathbb{Z}$.
\begin{definition}\label{d1}
A non-stationary dynamical system (or $nsds$) $(M, \langle . , .\rangle, F)$ is a mapping $F : M\rightarrow M$ such that, for each $i\in \mathbb{Z}$, $F\vert _{M_i} = f_i : M_i\rightarrow M_{i+1}$ is a $C^1$-diffeomorphism. We use the notation $F=(f_i)_{i\in \mathbb{Z}}$. The $n$-th composition of $F$ on $M_i$ is defined by
\begin{equation*}
F_i^n=\left\{
\begin{array}{rl}
& f_{i+n-1} \circ \cdots \circ f_i : M_i \rightarrow M_{i+n} ~~~\text{if } ~~~n>0\\
& f_{i+n}^{-1} \circ \cdots \circ f_{i-1} ^{-1}: M_i \rightarrow M_{i+n} ~~~\text{if } ~~~~~n<0\\
& I_i : M_i\longrightarrow M_{i}~~~~~~~~~~~~~~~~~~~~~~~~~~~~~~~ \text{if } ~~~~~~~~n=0,
\end{array} \right.
\end{equation*}
\end{definition}
\noindent where $I_i : M_i\rightarrow M_i$ is the identity on $M_i$. In this definition $f_i$ may not be an isometry and the Lorentzian metric vary on each $M_i$. We also use of the name $L$-family for the non-stationary dynamical system (or $nsds$) $(M, \langle . , .\rangle, F)$.
\begin{definition}\label{d12}
An $L$-family $(M, \langle . , .\rangle, F)$ is called a Lorentzian Anosov family ($LA$-family) up to a distribution $p\mapsto E^n(p)$, if there exist constants $0<\lambda <1$, $c>0$ and a continuous splitting $T_p M= E^s(p)\oplus E^u(p)\oplus E^n(p)$ for each $p\in M$ such that\\
$i)$ Each vector of $E^n(p)$ is a null vector and each non-zero vector in $E^s(p)$ or $E^u(p)$ is spacelike or timelike:\\
$ii)$ The splitting $T_p M=E^s(p)\oplus E^u(p)\oplus E^n(p)$ up to the sequence of distribution $p\mapsto E^n(p)$ is $D F$-invariant, i.e., for each $p\in M$, $D_p F(E^s_p)=E^s_{F(p)}$ and $D_p F(E^u_p)=E^u_{F(p)}$, where $T_p M$ is the tangent space at $p$;\\
$iii)$ For each $i\in \mathbb{Z}$, $n\in\mathbb{N}$ and $p\in M_i$, we have:\\
\begin{equation*}
\vert g_{F^n_i(p)} (D_p F^n_i(v), D_p F^n_i(v))\vert \leqslant c \lambda^n \vert g_p(v,v)\vert~~~~~~~~~~~~~~~~~~~~~~~~~~if~~~~~ v\in E^s(p)
\end{equation*}
and
\begin{equation*}
\vert g_{F^{-n}_i(p)} (D_p F^{-n}_i(v), D_p F^{-n}_i(v))\vert \leqslant c \lambda^n \vert g_p(v,v)\vert~~~~~if~~~~~ v\in E^u(p)
\end{equation*}
$(iv)$ For each $\xi\in E^u_p$ and for each $\nu\in T_p M$ with the property
$$\vert g_{F^{-n}_i(p)} (D_p F^{-n}_i(\nu), D_p F^{-n}_i(\nu))\vert \leqslant c \lambda^n \vert g_p(\nu,\nu)\vert$$ we have $\lim_{n\rightarrow\infty} g_{F_i^{-n}(p)} (D_p F^{-n}_i(\xi), D_p F^{-n}_i(\nu)) =0$.
\end{definition}
\noindent The subspaces $E^s(p)$ and $E^u(p)$ are called stable and unstable subspaces respectively.
\begin{lemma}\label{l1}
The statement
$\vert g_{F^n_i(p)} (D_p F^n_i(v), D_p F^n_i(v))\vert \leqslant c \lambda^n \vert g_p(v,v)\vert$ for each $i\in\mathbb{Z}$ and $n\in\mathbb{N}$ and for $v\in E^s_p$ is equivalent to $\vert g_{F^{-n}_i(p)} (D_p F^{-n}_i(v), D_p F^{-n}_i(v))\vert \geqslant c^{-1} \lambda^{-n} \vert g_p(v,v)\vert$ and the same is true for the condition on the unstable subspace.
\end{lemma}
 \begin{proof}
Let $w=D_p F_i^n(v)$ where $w\in E^s_{F_i^n(p)}$. Then
\begin{align*}
\vert g_{{F^n_i}(p)} (w, w)\vert
&=\vert g_{F^n_i(p)} (D_p F^n_i(v), D_p F^n_i(v))\vert\\
& \leqslant c \lambda^n \vert g_p(v,v)\vert \\
&=c \lambda^n \vert g_p(D_{F^n_i(p)} F^{-n}_i (w), D_{F^n_i(p)} F^{-n}_i (w))\vert
\end{align*}
If we put $F^n_i(p)=q$,
then $\vert g_{F^{-n}_i(q)} (D_{q} F^{-n}_i(w), D_{q} F^{-n}_i(w))\vert \geqslant c^{-1} \lambda^{-n} \vert g_q(w,w)\vert$. Since $D_p F_i$ is an isomorphism then the inequality is true for each element of $E^s_p$.
\end{proof}
 \begin{remark}\label{r1}
Base on the previous lemma the condition $(iii)$ of definition \ref{d12} can be replaced by:\\
$iii')$ $\vert g_{F^n_i(p)} (D_p F^n_i(v), D_p F^n_i(v))\vert \leqslant c \lambda^n \vert g_p(v,v)\vert ~~if~~ v\in E^s(p)$ and\\
$\vert g_{F^{n}_i(p)} (D_p F^{n}_i(v), D_p F^{n}_i(v))\vert \geqslant c^{-1} \lambda^{-n} \vert g_p(v,v)\vert~~if~~ v\in E^u(p)$\\
$iii'')$ $\vert g_{F^{-n}_i(p)} (D_p F^{-n}_i(v), D_p F^{-n}_i(v))\vert \geqslant c^{-1} \lambda^{-n} \vert g_p(v,v)\vert ~~if~~ v\in E^s(p)$ and\\
$\vert g_{F^{-n}_i(p)} (D_p F^{-n}_i(v), D_p F^{-n}_i(v))\vert \leqslant c \lambda^n \vert g_p(v,v)\vert~~if~~ v\in E^u(p)$
\end{remark}
\begin{proposition}\label{p1}
Given an $LA$-family $(M,F)$ up to a sequence of distributions $p\mapsto E^n(p)$. Then for each $p\in M$, the tangent space of $M$ at $p$ has a unique splitting.
\end{proposition}
\begin{proof}
Since the splitting is invariant then it is determined by the splitting on each component, so we restrict our proof on $M_0$. Let $p\in M_0$ be given and let $T_p M_0=E^s(p)\oplus E^u(p)\oplus E^n(p)=\tilde{E}^s(p)\oplus\tilde{E}^u(p)\oplus E^n(p)$ up to the distribution $p\mapsto {E}^n(p)$. Since $E^s(p)\oplus E^u(p) =\tilde{E}^s(p)\oplus\tilde{E}^u(p)$ then it is enough to prove that $E^u=\tilde{E}^u$. If $\xi\in E^u(p)$, then $\xi=\nu+\omega$, where $\nu\in \tilde{E}^u(p)$ and $\omega\in \tilde{E}^s(p)$. We show that the vector $\omega$ is a null vector. By lemma \ref {l1} we have
\begin{align*}
c^{-1} \lambda^{-n} \vert g_p(\omega,\omega)\vert &\leqslant \vert g_{F^{-n}_i(p)} (D_p F^{-n}_i(\omega), D_p F^{-n}_i(\omega))\vert\\
&= \vert g_{F^{-n}_i(p)} (D_p F^{-n}_i(\xi-\nu), D_p F^{-n}_i(\xi-\nu))\vert\\
&=\vert g_{F^{-n}_i(p)} (D_p F^{-n}_i(\xi ), D_p F^{-n}_i(\xi))\\
&+g_{F^{-n}_i(p)} (D_p F^{-n}_i(\nu), D_p F^{-n}_i(\nu))\\
& -2 g_{F^{-n}_i(p)} (D_p F^{-n}_i(\xi), D_p F^{-n}_i(\nu))\vert\\
&\leqslant c \lambda^n ( \vert g_p(\xi, \xi) \vert +\vert g_p(\nu, \nu)\vert)\\
&+2\vert g_{F^{-n}_i(p)} (D_p F^{-n}_i(\xi), D_p F^{-n}_i(\nu))\vert
\end{align*}
Since $0<\lambda<1$, then the above inequality tends to zero as $n$ tends to infinity. Hence the vector $\omega$ is a null vector. So $E^u(p)\subseteq \tilde{E}^u(p)$. We can show by similar calculation $ \tilde{E}^u(p)\subseteq E^u(p)$, therefore $E^u(p)=\tilde{E}^u(p)$ and the splitting is unique.
\end{proof}
%\begin{remark}
%If in Proposition \ref{p1} we suppose $\xi\in E^s_p$ and for each $\nu\in T_p M$ with the property $\vert g_{F^{n}_i(p)} (D_p F^{n}_i(\nu), D_p F^{n}_i(\nu))\vert \leqslant c \lambda^n \vert g_p(\nu,\nu)\vert$ such that \\$\lim_{n\rightarrow\infty} g_{F_i^{n}(p)} (D_p F^{n}_i(\xi), D_p F^{n}_i(\nu)) =0$ then the splitting in unique.
%\end{remark}
 \begin{theorem}\label{t1}
Let $(M,F)$ be an $LA$-family up to a $d$ dimensional distribution $p\mapsto E^n(p)$. Let $\gamma$ be a curve with $\gamma (t_n)\in M$ such that $t_n\rightarrow 0$ and $\gamma (t_n)\rightarrow p\in M$ when $n\rightarrow \infty$, then for a subsequence $\{\zeta _n\}$ of $\{t_n\}$, which we call it again $\{t_n\}$, we have
$$E^u_{(\gamma(t_n))}\rightarrow E^u(p) ~~\text{when}~~~ n\rightarrow\infty,$$
and
$$E^s_{(\gamma(t_n))}\rightarrow E^s(p)~~ \text{when}~~~ n\rightarrow\infty.$$
\end{theorem}
\begin{proof}
We take a point $p=(i,m_i)\in M$ with $m_i\in M_i$ and the map $F : M\rightarrow M$ with $F(i,m_i)=(i+1, f_i(m_i))$, $f_i(m_i)\in M_{i+1}$ and a curve $\gamma(t_n)$ on $M$ such that $t_n\rightarrow 0$ when $n\rightarrow\infty$. In fact $\gamma (t_n)=(i, \gamma_i(t_n))$ where $\gamma_i(t_n)\in M_i$. First we prove $d(B_{E^u(\gamma(t_n))} , B_{E^u(p)})\rightarrow 0$ when $n\rightarrow\infty$ where $B_{E^u(\gamma(t_n))}$ is a basis for $E^u_{(\gamma(t_n))}$ and $B_{E^u(p)}$ is a basis for $E^u(p)$. Second we show the convergence of the vectors in $E^u_{(\gamma(t_n))}$ to the vectors in $E^u(p)$.\\
Let $m$ be the dimensional of $M$. Then $0\leqslant dim (E^u(\gamma(t_n)))\leqslant m$ for all $n\in \mathbb{N}$. There exsit a constant $l$ and a subsequence $\{\zeta _n\in [-\epsilon/2, \epsilon/2] : n\in \mathbb{N}\}$ such that dim $(E^u(\gamma(\zeta _n))=l$ for all $n\in\mathbb{N}$. We take an orthonormal basis $B_{E^u(\gamma(\zeta _1))}=\{v_{11}, ... , v_{1l}\}$ and we translate it parallelly via linear isomorphism $P_t$ to $B_{E^u(\gamma(\zeta _n))}$ by $\{ v_{n1}=P_{\zeta _n-\zeta _1}(v_{11}), v_{n2}=P_{\zeta _n-\zeta _1}(v_{12}), ... , v_{nl}=P_{\zeta _n-\zeta _1}(v_{1l})\}$ that is an orthonormal basis for $E^u(\gamma(\zeta _n))$. The sequence $\{v_{nj}\}$ is a convergent sequence in the tangent bundle $TM$. This limit is $v_{j}=lim_{n\rightarrow\infty}v_{nj}= lim_{n\rightarrow\infty} P_{\zeta _n- \zeta _1}(v_{1j})=P_{lim_{n\rightarrow \infty}\zeta _n- \zeta_1}(v_{1j})$. Obviously $v_j\in E^s(p)\oplus E^u(p)$, thus $v_j=u+w$ where $u\in E^s(p)$ and $w\in E^u(p)$. By using of Lemma \ref{l1} we prove $u=0$. In fact we have\\
$c^{-1} \lambda^{-n} \vert g_p(u,u)\vert$\\
$\leqslant \vert g_{F^{-n}_i(p)} (D_p F^{-n}_i(u), D_p F^{-n}_i(u))\vert$ \\
$ =\vert g_{F^{-n}_i(p)} (D_p F^{-n}_i(u+w-w), D_p F^{-n}_i(u+w-w))\vert $\\
$\leqslant \vert g_{F^{-n}_i(p)} (D_p F^{-n}_i(v_j), D_p F^{-n}_i(v_j))\vert $\\
$ + \vert g_{F^{-n}_i(p)} (D_p F^{-n}_i(w), D_p F^{-n}_i(w))\vert $\\
$+2 \vert g_{F^{-n}_i(p)} (D_p F^{-n}_i(v_j), D_p F^{-n}_i (w)) \vert$\\
$\leqslant lim_{t_r\rightarrow 0} \mid g_{F^{-n}_i(p)} (D_p F^{-n}_i(v_{rj}), D_p F^{-n}_i(v_{rj})) \mid $\\
$+\mid g_{F^{-n}_i(p)} (D_p F^{-n}_i(w), D_p F^{-n}_i(w))\mid$\\
$+ 2 ~lim_{t_r\rightarrow 0}\mid g_{F^{-n}_i(p)} (D_p F^{-n}_i(v_{rj}), D_p F^{-n}_i(P_{t_r}w))\mid$\\
$ \leqslant lim_{t_r\rightarrow 0} c \lambda ^n \vert g_{F^{-n}_i(p)} (v_{rj}, v_{rj})\vert+ c \lambda^n \mid g_{F^{-n}_i(p)} (w, w) \mid$ \\ $+2~ lim_{t_r\rightarrow 0} \vert g_{F^{-n}_i(p)} ( D_p F^{-n}(v_{rj}), D_p F^{-n} (P_{t_r}w)\vert$. \\
Base on Definition \ref{d12} $lim_{t_r\rightarrow 0} \vert g_{F^{-n}_i(p)} ( D_p F^{-n}(v_{rj}), D_p F^{-n} (P_{t_r}w)\vert=0$. Thus $\vert g_p(u, u)\vert=0$. Hence $u=0$ or $u$ is a null vector. Since $u\in E^s(p)$ then $u=0$. Therefore the set $\{v_1, ..., v_l\}$ is an orthonormal subset of $E^u(p)$. Thus dim $(E^u(p))\geqslant l$. By the same calculations dim $(E^s(p))\geqslant m-l-d$ where the dimension of $E^n(p)$ is $d$. Hence dim $(E^u(p))=l$ and dim $(E^s(p))= m-l-d$. Thus with choosing sufficiently large $n$, dim $(E^u(\gamma (t_n))=l$ and dim $(E^s(\gamma (t_n))=m-l-d$. For the second part we consider an arbitrary unit vector $u_k=\sum_{j=1}^l \alpha_j v_{kj}\in E^u(\gamma(t_k))$, we have $\sum_{j=1}^l \alpha_j^2=1$. Since $v_{kj}\rightarrow v_j$ then $\vert g_p(\sum_{j=1}^l \alpha_j v_{j} , \sum_{j=1}^l \alpha_j v_{j}) \vert=1$ and $lim_{{t_k}\rightarrow 0} \vert g_p(P_{-t_k}(v_{kj}) -v_j, P_{-t_k}(v_{kj'}) -v_{j'} ) \vert =0$. Last equality implies for given $\delta>0$ there is $M>0$ such that for all $k>M$ and $j, j'\in \{1, ... , l\}$ we have
$$\vert g_p(P_{-t_k}(v_{kj}) -v_j, P_{-t_k}(v_{kj'}) -v_{j'} ) \vert < \delta / (1+\mid \sum_{j=1}^l \sum_{j'=1}^l \alpha_j \alpha_j' \mid).$$
Therefore\\
$\mid g_p (P_{-t_k}(u_k)-\sum_{j=1}^l \alpha _j v_j , P_{-t_k}(u_k)-\sum_{j=1}^l \alpha _j v_j ) \mid$\\
$=\vert g_p (\sum_{j=1}^l \alpha _j P_{-t_k}(v_{k j})-\sum_{j=1}^l \alpha _j v_j , \sum_{j=1}^l \alpha _j P_{-t_k}(v_{k j})-\sum_{j=1}^l \alpha _j v_j ) \mid $\\
$=\vert \sum_{j=1}^l \sum_{j'=1}^l \alpha_j \alpha_{j'} g_p (P_{-t_k}(v_{kj})-v_j, P_{-t_k}(v_{kj'})-v_{j'} )\vert $\\
$ \leqslant \delta \mid \sum_{j=1}^l \sum_{j'=1}^l \alpha_j \alpha_{j'} \mid / (1+\mid \sum_{j=1}^l \sum_{j'=1}^l \alpha_j \alpha_{j'} \mid)\leqslant \delta~~~~~~(*)$

Now we take an arbitrary unit vector $v=\sum_{j=1}^l \alpha_j v_j \in E^u(p)$. Since $v_{nj}\rightarrow v_j$ then $\vert g_{\gamma(t_n)} (\sum_{j=1} ^ l\alpha_j v_{nj} , \sum_{j=1} ^ l\alpha_j v_{nj} )\vert=1$. Thus for given $\delta>0$ we have
$$\vert g_{\gamma(t_n)}(P_{t_n}(v_{j}) -v_{nj}, P_{t_n}(v_{j'}) -v_{nj'} ) \vert < \delta / (1+\mid \sum_{j=1}^l \sum_{j'=1}^l \alpha_j \alpha_{j'} \mid).$$
Consequently\\
$\mid g_{\gamma(t_n)} (P_{t_n}(v)-\sum_{j=1}^l \alpha _j v_{nj} , P_{t_n}(v)-\sum_{j=1}^l \alpha _j v_{nj} ) \mid$\\
$=\vert g_{\gamma(t_n)} (\sum_{j=1}^l \alpha _j P_{t_n}(v_{j})-\sum_{j=1}^l \alpha _j v_{nj} , \sum_{j=1}^l \alpha _j P_{t_n}(v_{j})-\sum_{j=1}^l \alpha _j v_{nj}) \mid $\\
$=\vert \sum_{j=1}^l \sum_{j'=1}^l \alpha_j \alpha_{j'} g_{\gamma(t_n)} (P_{t_n}(v_{j})-v_{nj}, P_{t_n}(v_{j'})-v_{nj'}) \vert$\\
$\leqslant \delta \mid \sum_{j=1}^l \sum_{j'=1}^l \alpha_j \alpha_{j'} \mid / (1+\mid \sum_{j=1}^l \sum_{j'=1}^l \alpha_j \alpha_{j'} \mid)\leqslant \delta~~~~~~(**)$
The inequalities $(*)$ and $(**)$ implies $E^u(\gamma(t_n))$ convergence to $E^u(p)$. The second part of Theorem can prove similarly.
\end{proof}
\section{Examples of Lorentzian non-stationary dynamical system}
Let $M_0$ be a smooth Lorentzian manifold with a metric $g_0$ which has the splitting $TM_0=E_0^s\oplus E_0^u\oplus E_0^n$ for each $p\in M_0$. We make an $LA$-family $(M, \langle . , .\rangle, F)$ by taking $M=\displaystyle\bigcup_{i\in\mathbb{Z}}^0 M_i$, where $M_i=\{i\}\times M_0$ and we define a Lorentzian metric on each $M_i$ by
$g_i\vert _{E^s}=\alpha ^ {-\vert i \vert} g_0 \vert _{E^s},~~~g_i\vert _{E^u}=\alpha ^ {\vert i\vert} g_0 \vert _{E^u},~~~ \text{and} ~~~g_i\vert _{E^n}=g_0 \vert _{E^n}$, where $i\neq 0$ and $\alpha>1$. In fact we contract $g_0$ exponentially along the subspace $E_0^s$ and we expand it along the subspace $E^u_0$. If we take $f_i : M_i\rightarrow M_{i+1}$ with the properties $Df_i(E^s)=E^s,~~~Df_i(E^u)=E^u~~~ \text{and}~~~ Df_i(E^n)=E^n$, then $(M, \langle . , .\rangle, F=\{f_i\})$ is an $LA$-family.\\ Now we present another example.
\begin{example}
Take the solid torus $N=S^1 \times D^2$ where $D^2$ is the unit disk in $\mathbb{R}^2$. The coordinates of this manifold are of the forms $(\theta, u, v)$ such that $ \theta\in S^1$ and $(u, v)\in D^2$, with $u^2+v^2\leqslant 1$. Using these coordinate the dynamical system $f: N\rightarrow N$ defined by
$$f(\theta, u , v)=(2\theta, \dfrac{1}{10}u +\dfrac{1}{10} cos \theta, \dfrac{1}{10}v +\dfrac{1}{10} sin \theta)$$
creates a solenoid or Smale attractor by the iterations of $f$. In fact it is the maximal invariant hyperbolic set $\Lambda:= N\bigcap f(N)\bigcap f^2(N) ... =\displaystyle\bigcap _{n\in\mathbb{N}} f^n(N)$ \cite{n12}.
If we restrict $f$ on $\Lambda$ then we have an Anosov map in the sense of Riemannian metric. We define a mapping $\tilde{f} : \Lambda \times \mathbb{R} \subset \mathbb{R}^4\rightarrow \Lambda \times \mathbb{R} \subset \mathbb{R}^4$ by $\tilde{f}(\theta, u , v , z)=(f(\theta, u , v) , z)$. On $\Lambda \times \mathbb{R}$ we take the induced metric of $\mathbb{R}^4$ defined by $g_p(X, Y)= \theta_1 \theta_2 +u_1 u_2 +v_1 v_2-z_1 z_2$ where $p\in\mathbb{R}^4$ and $X=(\theta_1, u_1, v_1, z_1), ~Y=(\theta_2, u_2, v_2, z_2) \in T_p \mathbb{R}^4$. If we take $E^n(p)=\{(a, a, a, \sqrt{3} a) : a\in\mathbb{R}\}$ then $\Lambda \times \mathbb{R}$ is a Lorentzian hyperbolic set for $\tilde{f}$ and it is a Lorentzian Anosov map. We define an $LA$-family $(M, F)$ by taking the components of $M$ i.e. $M_i$ as distinct copies of $\Lambda \times \mathbb{R}$ and the mapping $\tilde{f_i}: M_i\rightarrow M_{i+1}$ by $(i,x)\mapsto (i+1, \tilde{f}_i(x))$, for $i\in\mathbb{Z}$ (see figure $1$).

 \begin{figure}
\centering
\includegraphics[scale=0.8]{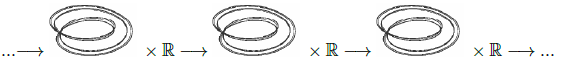}
\caption{The Anosov family when $\Lambda$ is approximated by $ N\bigcap f(N)\bigcap f^2(N)$.}
\end{figure}

 \end{example}
 \begin{example}
Let $(M, g)$ be a connected two dimensional $C^k$ Riemannian manifold. We take $M_i=I_i \times M$ where $I_i=(0, 2^i)$ for each $i\in\mathbb{Z}$.\\ By putting a Lorentzian metric $g_i (v_{1i}\oplus w_1, v_{2i}\oplus w_2)=-v_{1i} v_{2i}+ d^{\mid i\mid} g(w_1, w_2)$ where $0<d<1$, $v_{1i} \oplus w_2, ~v_{2i}\oplus w_2\in T_{(t,p)} (I\times M)$, each $M_i$ is a warped product manifold. We make an $LA$-family by defining $f_i : I_i\times M\rightarrow I_{i+1}\times M$, $(x, y)\mapsto (2x, h(y))$ where $h$ is a diffeomorphism on $M$. If we take \\
$E^n_{ip}=<(d^{\mid i \mid} \sqrt{g((0,1), (0,1)) },0 ,1)>$, $E^u_{ip}=<(1, 0, 0)>$ and $E^s_{ip}=<(0, 1, 1)>$, then $(M, \{f_i\})$ is an $LA$-family.
\end{example}

\section{Shadowing property on Lorentzian non-stationary dynamical system}
In this section we define Lorentzian shadowing property of type $I$ and type $II$ on $L$-family $(M, \langle . , .\rangle, F)$ and on the induced dynamic by it. We prove shadowing property of type $I$ is invariant by uniformly conjugacy and if $L$-family $(M, \langle . , .\rangle, F)$ has shadowing property of type $ II $ then it has shadowing property of type $ I $.\\

\begin{definition}\label{d1}
Consider a $L$-family $(M, \langle . , .\rangle, F)$ for $ \delta>0 $, the sequence $ \{v_i\}_{i\in \mathbb{Z}} $ where $v_i\in T_p M_i$ is said to be a $\delta$-pseudo orbit of type $I$ if $ \vert g_{i+1} (D_p f_i (v_i), v_{i+1})\vert < \delta $, for $i\in \mathbb{Z}$.
\end{definition}
For given $\epsilon>0$ a $ \delta $-pseudo orbit $ \{v_i\}_{i\in\mathbb{Z}} $ is called to be $ \epsilon $-traced by $ w\in T_p M_i $
if $ \vert g_{i+n}(F_i^n(w), v_{i+n})) \vert <\epsilon $ for $n\in \mathbb{Z}$.
\begin{definition}\label{d2}
A $L$-family $(M, \langle . , .\rangle, F)$ is said to have Lorentzian shadowing property of type $I$ if for every $\epsilon>0$, there exists a $ \delta>0 $ such that every $ \delta $-pseudo orbit is $ \epsilon $-traced by some vector in $ T_pM_i $.\
\end{definition}
Consider two $L$-family $(M, g, F)$ and $(\tilde{M}, \tilde{g} , \tilde{F})$. We say a homeomorphism $h : M\rightarrow \tilde{M}$ is uniform continuous if for given $ \epsilon >0 $ there exists an $\tilde{\epsilon} > 0$ such that $ \vert g(v,w)\vert <\epsilon $ implies $ \vert \tilde{g}( h_i(v), h_i(w)) \vert < \tilde{\epsilon}$ where $v, w \in T_pM_i$ and $i\in \mathbb{Z}$.
\begin{definition}\label{u1}
A topological conjugacy between $L$-family $(M, g, F)$ and $(\tilde{M}, \tilde{g}, \tilde{F})$ is a map $ h: M\rightarrow \tilde{M}$ such that for $i\in\mathbb{Z}$, $ h \vert _{M_i}=h_i : M_i\rightarrow \tilde{M}_i$ is a homeomorphism and $ h_{i+1} \circ f_i= \tilde{f}_i \circ h_i $.
\end{definition}
\begin{definition}
A topological conjugacy $h : M\rightarrow \tilde{M}$ is uniformly conjugate if $ h_i : M_i\rightarrow \tilde{M}_i$ and $ h_i^{-1} : \tilde{M}_i\rightarrow M_i$ are uniformly continuous.
\end{definition}
%Now we want to define the product of two $L$-family $(M, g, F)$ and $(\tilde{M}, \tilde{g}, \tilde{F})$. we put $M \times \tilde{M}=\bar{M}$, $ g \times\tilde{g} =\bar {g}$. It means $M_i \times \tilde{M}_i = \bar{M}_i$, $g_i \oplus \tilde{g}_i=\bar{g}_i$ where $\bar{g}_i (v_i +\tilde{v}_i, w_i +\tilde{w}_i) = g_i(v_i, w_i) \oplus \tilde{g}_i (\tilde{v}_i, \tilde{w}_i)$ for all vectors $v_i, w_i \in T_p M_i$, $\tilde{v}_i, \tilde{w}_i \in T_{\tilde{p}} \tilde{M}_i$
%and $f_i \times \tilde{f}_i= \bar{f}_i$ for $i\in \mathbb{Z}$.

 \begin{theorem}\label{u2}
Let $L$-family $(M, g, F)$ be uniformly conjugate $(\tilde{M}, \tilde{g}, \tilde{F})$. Then $(M, g, F)$ has shadowing property of type $I$ if and only if $(\tilde{M}, \tilde{g}, \tilde{F})$ has shadowing property of type $I$.\\
%b) If $L$-family $(M, g, F)$ and $(\tilde{M}, \tilde{g}, \tilde{F})$ have shadowing property of type $I$ if and only if product of them has shadowing property of type $I$.
\end{theorem}
\begin{proof}

Suppose $L$-family $(M, g, F)$ has shadowing property of type $I$. Given $\epsilon>0$, because of uniform continuity of $h$ there exists $0< \epsilon_0< \epsilon$ such that $ \vert g(u_i, v_i) \vert<\epsilon_0 $ implies $ \vert \tilde{g} (h_i(u_i), h_i(v_i))\vert < \epsilon$ for all $u_i, v_i\in T_p M_i$. Since $L$-family $(M, g, F)$ has shadowing property of type $I$ there exists $0< \delta_0 <\epsilon_0$ such that for every $\delta_0$-pseudo orbit of $L$-family $(M, g, F)$ is $\epsilon_0$-traced by $w\in T_pM_i$. By uniform continuous of $h^{-1}$, there exists $0<\delta<\delta_0$ such that $ \vert \tilde{g}(\tilde{u}_i, \tilde{v}_i)\vert<\delta $ implies $ g(h_i^{-1}(\tilde{u}_i), h_i^{-1}(\tilde{v}_i)) < \delta_0$ for any vector $\tilde{u}_i, \tilde{v}_i\in T_{\tilde{p}}\tilde{M}_i$. Now we prove that every $\delta$-pseudo orbit of $L$-family $(\tilde{M}, \tilde{g}, \tilde{F})$ is $\epsilon$-traced by some vector of $T_{\tilde{p}}\tilde{M}_i$. Suppose $\{ \tilde{v}_i\}$ is a $\delta$-pseudo orbit of $L$-family $(\tilde{M}, \tilde{g}, \tilde{F})$ i.e
$ \vert \tilde{g}_{i+1}(\tilde{f} _i (\tilde{v}_i), \tilde{v}_{i+1})\vert < \delta$ so $ \vert g_{i+1}(h_{i+1}^{-1}\tilde{f} _i (\tilde{v}_i), h_{i+1}^{-1}(\tilde{v}_{i+1}))\vert < \delta_0$ then $ \vert g_{i+1}(f_i\circ h^{-1}_i (\tilde{v}_i), h_{i+1}^{-1}(\tilde{v}_{i+1}))\vert < \delta_0$. Hence $\{h_i^{-1}(\tilde{v}_i)\}$ is a $ \delta_0 $-pseudo orbit for $L$-family $(M, g, F)$. Thus there exist $w\in T_p M_i$ such that
$ \vert g_{i+n} (F_i^n(w), h_{i+n}^{-1} (\tilde{v}_{i+n})\vert <\epsilon _0 $ for $ n\in \mathbb{Z} $ . So by using uniform continuous of $h$ we have $ \vert\tilde{g}_{i+n} (h_{i+n}F_i^n(w), \tilde{v}_{i+n}\vert <\epsilon$ for $n\in \mathbb{Z}$. Because $ h_{i+n} \circ F_i^n(w)=\tilde{F}_i^n \circ h_i(w) $ so $ \vert\tilde{g}_{i+n}(\tilde{F}_i^n(\tilde{w}), \tilde{v}_{i+n}\vert <\epsilon $ for $n\in \mathbb{Z}$ where $h_i (w)=\tilde{w}\in T_{\tilde{p}}\tilde{M}_i$. Hence $\{\tilde{v}_i\}$ is $\epsilon$- traced by $\tilde{w}$. The converse is proved by similar argument.\\
\end{proof}
For the next theorem we define a metric on all non-empty compact subspaces of $TM=\cup_{i\in\mathbb{Z}} TM_i$. Because $T_p M_i\cong \mathbb{R}^n$ where dim $M_i=n$ we put usual topology on $T_pM_i$. Suppose $\mathcal{K}(M_i)$ denotes the hyperspace of all non-empty compact subspaces of $T_{p }M_i$. Again consider $\gamma: (-\epsilon, \epsilon)\rightarrow M_i$ is a smooth curve passing throught $p$. If $u\in T_{\gamma(t)} M_i$ and $K$ is a non-empty compact subspace of $T_{\gamma(\zeta)} M_i$ then $d(u, K)= \inf \{\vert g_{\gamma(\zeta)} ( P_{\zeta-t}(u),v ) \vert : v\in K\}$ where $t, \zeta\in (-\epsilon, \epsilon)$. So for two non-empty compact subspaces $A$ and $B$ of $T_{\gamma(t)} M_i$ and $T_{\gamma(\zeta)} M_i$ we define $d(A, B)=max \{max~~ d(u, B) , max~~ d(v, A) : u\in A ~~~and~~~v\in B \}$. Consider homeomorphism $D_p f_i : T_pM_i\rightarrow T_{f_i(p)}M_{i+1}$ for $i\in\mathbb{Z}$, it induces a homeomorphism $ \mathcal{F} : \mathcal{K}(M_i)\rightarrow \mathcal{K}(M_{i+1})$ by $\mathcal{F}(K)=D_p f_i(K)$ for every $K\in\mathcal{K}(M_i)$ where $\mathcal{F}(K)=\{D_p f_i(k) : k\in K\}$.
Now we extend definitions \ref{d1} and \ref{d2} on $\mathcal{K}(M)=\cup_{i\in\mathbb{Z}}\mathcal{K}(M_i)$\\
Consider a $L$-family $(M, \langle . , .\rangle, F)$ for $ \delta>0 $, the sequence $ \{V_i\}_{i\in \mathbb{Z}} $ where $V_i\in \mathcal{K} (M_i)$ is said to be a $\delta$-pseudo orbit of type $II$ if $ d (\mathcal{F}(V_i), V_{i+1})< \delta $, for $i\in \mathbb{Z}$.\\
For given $\epsilon>0$ a $ \delta $-pseudo orbit $ \{V_i\}_{i\in\mathbb{Z}} $ is called to be $ \epsilon $-traced by $ W\in\mathcal{K}(M_i )$
if $ d(\mathcal{F}_i^n(W), V_{i+n})) <\epsilon $ for $n\in\mathbb{Z}$. \\
A $L$-family $(M, \langle . , .\rangle, F)$ is said to have Lorentzian shadowing property of type $II$ if for every $\epsilon>0$, there exists a $ \delta>0 $ such that every $ \delta $-pseudo orbit is $ \epsilon $-traced by some vector in $\mathcal{K} (M_i) $.
\begin{theorem}\label{u3}
If $L$-family $(M, g, F)$ has shadowing property of type $ II $ then it has shadowing property of type $ I $.
\end{theorem}

 \begin{proof}
Suppose $L$-family $(M, g, F)$ has shadowing property of type $ II $. Let $\{v_i\}_{i\in\mathbb{Z}}$ be a $\delta$-pseudo orbit of type $I$ then $\{\{v_i\}\}_{i\in\mathbb{Z}}$ is a $\delta$-pseudo orbit of type $II$. For any $\epsilon>0$ a $ \delta $-pseudo orbit $\{ V_i=\{v_i\}\}_{i\in\mathbb{Z}} $ is $ \epsilon $-traced by $ W\in\mathcal{K}(M_i )$. \\ Then $ d(\mathcal{F}_i^n(W), V_{i+n}) <\epsilon $ for $n\in\mathbb{Z}$. We can write $\vert g_{i+n}(F_i^n(w), v_{i+n})\vert <\epsilon $ where $w\in W$ and $v_{i+n}\in V_{i+n}$ for $n\in \mathbb{Z}$. So $L$-family $(M, g, F)$ has shadowing property of type $ I $.
\end{proof}

 \section{Conclusion}
We extend the notion of Anosov family on a sequence of compact Riemannian manifolds $\{M_i\}_{i\in\mathbb{Z}} $ to Lorentzian Anosov family by using of a distributions $p\mapsto E^n(p)$. We see that when each $M_i$ is a Lorentzian manifold then the tangent space of $M_i$ at $p\in M_i$ has a unique splitting to stable, unstable and a subset of null vectors which is also a vector space. We consider the behavior of this splitting by using of the distance function created by the unique torsion-free connection. \\
%Hyperbolic behavior of the flows on pseudo-Riemannian manifolds has been considered in \cite{mm}. The extension of this notion to $LA$-family can be a topic for further research.\\
 Lorentzian shadowing property of type $I$ and type $II$ on $L$-family $(M, \langle . , .\rangle, F)$ and on the induced dynamic by it has been introduced. We prove shadowing property of type $I$ is invariant by uniformly conjugacy. Also we prove if $L$-family $(M, g, F)$ has shadowing property of type $ II $ then the induced dynamic of it has shadowing property of type $ I $.

\end{document}